\documentclass[11pt]{amsart}
\usepackage{color}
\usepackage[utf8]{inputenc}
\usepackage{amsmath,amsthm,amssymb,amscd}
\usepackage[all,cmtip]{xy}
\usepackage{booktabs} 
\usepackage[hyperfootnotes=false, colorlinks, linkcolor={blue}, citecolor={magenta}, filecolor={blue}, urlcolor={blue}, plainpages=false, pdfpagelabels]{hyperref}
\usepackage[overload]{textcase} 
\usepackage{tikz}
\usetikzlibrary{decorations.pathmorphing}
\usepackage[export]{adjustbox} 
\usepackage{tikz-cd}
\usepackage{mathtools}
\usepackage[numbers,sort&compress]{natbib} 
\usepackage[left=3cm, right=3cm, top=3.5cm, bottom=3cm]{geometry}
\usepackage{enumitem}  

\usepackage{float}
\usepackage{longtable} 
\usepackage[ampersand]{easylist}

\usepackage{comment}
\usepackage{ltablex,booktabs}
\usepackage{amsthm}
\usepackage{empheq}
\allowdisplaybreaks
\newcommand{\A}{{\mathbb A}}
\newcommand{\Q}{{\mathbb Q}}
\newcommand{\Z}{{\mathbb Z}}

\newcommand{\C}{{\mathbb C}}
\newcommand{\N}{{\mathbb N}}

\newcommand{\GL}{{\rm GL}}

\newcommand{\SL}{{\rm SL}}

\newcommand{\Sp}{{\rm Sp}}
\newcommand{\GSp}{{\rm GSp}}
\newcommand{\PGSp}{{\rm PGSp}}

\newcommand{\Kl}{\Gamma_0'}

\newcommand{\Hh}{{\mathcal H}}

\newcommand{\forget}[1]{}

\def\qdots{\mathinner{\mkern1mu\raise0pt\vbox{\kern7pt\hbox{.}}\mkern2mu
		\raise3.4pt\hbox{.}\mkern2mu\raise7pt\hbox{.}\mkern1mu}}

\usepackage{colonequals}
\newtheorem{lemma}{Lemma}[section]
\newtheorem{theorem}[lemma]{Theorem}
\newtheorem{corollary}[lemma]{Corollary}
\newtheorem{proposition}[lemma]{Proposition}
\newtheorem{definition}[lemma]{Definition}

\title[Siegel cusp forms and $4$-core partitions]{Congruences for dimensions of spaces of Siegel cusp forms and $4$-core partitions}
\author{Chiranjit Ray}
\address{Department of Mathematics, Harish-Chandra Research Institute, U.P. 211019, INDIA.}
\curraddr{}
\email{\tt chiranjitray.m@gmail.com}
\author{Manami Roy}
\address{Department of Mathematics, Fordham University, Bronx, New York 10458, USA.}
\curraddr{}
\email{\tt manami.roy.90@gmail.com}
\author{Shaoyun Yi}
\address{Department of Mathematics, University of South Carolina, Columbia, SC 29208, USA.}
\curraddr{}
\email{\tt yishaoyun926@gmail.com}

\begin{document}
	\begin{abstract}
		Using the relationship between Siegel cusp forms of degree $2$ and cuspidal automorphic representations of $\GSp(4,\A_{\Q})$, we derive some congruences involving dimensions of spaces of Siegel cusp forms of degree $2$ and the class number of $\Q(\sqrt{-p})$. We also obtain some congruences between the $4$-core partition function $c_4(n)$ and dimensions of spaces of Siegel cusp forms of degree $2$.
	\end{abstract}
	\maketitle
	\section{Introduction}
	There is a well known connection between Siegel modular forms of degree $2$ and automorphic representations of the adelic group $\GSp(4,\A_{\Q})$; for more details see \cite{AsgariSchmidt2001, Schmidt2017}.  Let $S_{k}(\Gamma_N)$ be the space of Siegel cusp forms of degree $2$, weight $k$, and level $N$ with respect to a congruence subgroup $\Gamma_N$ of $\Sp(4,\Q)$. Here we consider the following congruence subgroups: the full modular group $\Sp(4,\Z)$; the paramodular group ${\rm K}(N)$ of level $N$, the Klingen congruence subgroup $\Kl(N)$ of level $N$, and the Siegel congruence subgroup $\Gamma_0(N)$ of level $N$
	defined as follows, respectively.
		\begin{equation}\label{classical congrunece subgroups}
		\begin{split}
		&{\rm K}(N)=\begin{bsmallmatrix} \Z&N\Z&\Z&\Z\\ \Z&\Z&\Z&N^{-1}\Z\\ \Z&N\Z&\Z&\Z
				\\ N\Z&N\Z&N\Z&\Z\end{bsmallmatrix}\ \cap\  \Sp(4,\Q),\\
			&
			\Kl(N)=\begin{bsmallmatrix} \Z&N\Z&\Z&\Z\\\Z&\Z&\Z&\Z\\ \Z&N\Z&\Z&\Z
				\\ N\Z&N\Z&N\Z&\Z\end{bsmallmatrix} \ \cap\  \Sp(4,\Z),\\
			&
			\Gamma_0(N)=\begin{bsmallmatrix} \Z&\Z&\Z&\Z\\\Z&\Z&\Z&\Z\\ N\Z&N\Z&\Z&\Z
				\\ N\Z&N\Z&\Z&\Z\end{bsmallmatrix} \ \cap\  \Sp(4,\Z).\\
		\end{split}
	\end{equation}

	\begin{definition}\label{Definition of S k p}
			Let $k\ge 3$ be a positive integer, and let $p$ be a prime. 
			Let $S_k(p,\Omega)$ be the set of cuspidal automorphic representations $\pi\cong\otimes_{v \le \infty} \pi_v$ of $\GSp(4,\A_{\Q})$ with trivial central character satisfying the following properties:
		\begin{enumerate}
			\item[\rm{(i)}]   $\pi_{\infty}$ is the lowest weight module with minimal $K$-type $(k,k)$; it is a holomorphic discrete series representation.
			\item[\rm{(ii)}]     $\pi_{v}$ is unramified for each $v\neq p, \infty$.
			\item[\rm{(iii)}]    $\pi_{p}$ is an Iwahori-spherical representation of $\GSp(4,\Q_p)$ of type $\Omega$.
		\end{enumerate}
We denote $s_k(p,\Omega)\colonequals \# S_k(p,\Omega)$.
	\end{definition}
	Here the representation type $\Omega$ is one of the types listed in \cite[Table 3]{Schmidt200501}: 	I, IIa, IIb, IIIa, \ldots, VId.
	The dimensions of spaces $S_k(\Gamma_p)$ of Siegel cusp forms of degree $2$ and level $p$ with respect to the congruence subgroups in \eqref{classical congrunece subgroups} are connected to the numbers $s_k(p,\Omega)$; see \cite{RoySchmidtYi2021}. 	We discuss this relationship briefly in Section~\ref{Preliminaries}. 
	In this article, we obtain some congruences between dimensions of spaces $S_k(\Gamma_p)$ of Siegel cusp forms and class numbers of imaginary quadratic fields $\Q(\sqrt{-p})$  utilizing the explicit formulas between dimensions of  $S_k(\Gamma_p)$ and the quantities $s_k(p,\Omega)$.

	The connection between class numbers of imaginary quadratic fields and dimensions of spaces of cusp forms may seem surprising, but it has been studied before. In \cite{Wakatsuki2014}, Wakatasuki proved some congruences between class number  of imaginary quadratic fields and dimensions of spaces of elliptic cusp forms. The congruence for elliptic cusp forms case follows from Yamauchi's formula \cite{Yamauchi1973} for traces of Atkin–Lehner involutions involving the class number of $\Q(\sqrt{-p})$. In the same paper, Wakatsuki also obtained some congruences between class number of imaginary quadratic fields $\Q(\sqrt{-p})$ and dimensions of spaces of vector-valued Siegel cusp forms. 
	Here, we prove some congruences for scalar-valued Siegel cusp forms in Section~\ref{main results}; to be more precise, see Theorem~\ref{Proposition of IIIa/VI p geq 5} and Theorem~\ref{Proposition of Va p geq 5}.
	
	The class number of imaginary quadratic fields appears in the study of certain partition functions, like $t$-core partitions. For positive integers $t$, we let $c_t(n)$ denote the number of $t$-core partitions of $n$. The arithmetic properties of $c_t(n)$ has been of interest in combinatorial number theory and representation theory; for example see \cite{Garvan1990, Hirschhorn1996JNT}. 
	The $4$-core partitions arise naturally in the modular representation theory of finite general linear groups. In \cite{Ono1997}, Ono and Sze studied $4$-cores partitions using Gauss' theory of class numbers. In Section~\ref{Section 4core}, we discuss some new congruences on $4$-core partitions. By the formula that connects class numbers of imaginary quadratic fields and $4$-core partitions, we obtain some congruences between $4$-core partitions and dimensions of spaces of Siegel cusp forms modulo $4$ in Corollary~\ref{Theorem for C4n and Smf}. As a consequence, we prove that for positive integer $n$ such that $8n+5$ is a prime number and for any positive integer $k$,
	\begin{equation}\label{eq in intro}
		\begin{split}
			c_4(n) \equiv \dim_{\C} S_{4k}^{\rm new}(\Gamma_{0}^{(1)}(8n+5)) \pmod{2}.
		\end{split}
	\end{equation}
Here  $S_{4k}^{\rm new}(\Gamma_{0}^{(1)}(8n+5))$ denotes the space of elliptic newforms of weight $k$  and level $8n+5$ with respect to the congruence subgroup $\Gamma_{0}^{(1)}(N)=\left[\begin{smallmatrix}
		\Z&\Z\\N\Z&\Z\\
	\end{smallmatrix}\right]\cap \SL(2,\Z)$.
	One can readily obtain from \eqref{eq in intro} that $c_4(n)$ is odd when $8n+5$ is a prime number; this result is also proven in \cite[Remark 1]{Ono1997} using Gauss's genus theory.
	\section{Preliminaries}\label{Preliminaries}
	In this section we review some basic results about modular forms, Siegel modular forms and the relationship of Siegel modular forms with automorphic representations of $\GSp(4,\A_{\Q})$. The algebraic group $\GSp(4)$ is defined by
	\begin{equation}\label{GSp4 classical version}
		\GSp(4) \colonequals \{g\in\GL(4)\colon \:^tgJg=\lambda(g)J,\:\lambda(g)\in \GL(1)\}, \quad J=\begin{bsmallmatrix} 0&1_2\\-1_2&0\end{bsmallmatrix}.
	\end{equation}
	The function $\lambda$ is called the multiplier homomorphism. The kernel of this function is the symplectic group $\Sp(4)$. Let $Z$ be the center of $\GSp(4)$ and $\PGSp(4)=\GSp(4)/Z$.
	
	Let $\Hh_2$ be the Siegel upper half space of degree $2$, i.e., $\Hh_2$ consists of all symmetric complex $2\times2$ matrices whose imaginary part is positive definite.
	\begin{definition}
			A Siegel modular form $f:\Hh_2\rightarrow \C$ of degree $2$ and weight $k$ with respect to a congruence subgroup $\Gamma_N$ of $\Sp(4,\Q)$ is a holomorphic function satisfying the following transformation property
		\begin{equation*}
			(f|_{k}g)(Z)=\det(CZ+D)^{-k}f((AZ+B)(CZ+D)^{-1})=f(Z)\ \text{for }  g=\begin{bsmallmatrix}A&B\\C&D\end{bsmallmatrix}  \in \Gamma_N.
		\end{equation*}
	\end{definition}
\textcolor{black}{Let  $\Gamma_N$ be one of the congruence subgroups given in \eqref{classical congrunece subgroups}}. We call a Siegel modular form $f$ a \emph{cusp form} if 
\begin{equation*}
	\lim\limits_{\lambda \rightarrow \infty} (f|_{k}g)\left(\begin{bsmallmatrix}i\lambda&\\&\tau\end{bsmallmatrix} \right)=0\quad \text{for all } g \in \Sp(4,\Q), \tau \in \Hh,
\end{equation*}
where $\Hh$ is the complex upper half plane. Now, let $f\in S_k(\Gamma_p)$ be an eigenform. We say that $f$ gives rise to an irreducible representation if its adelization
generates an irreducible cuspidal automorphic representation $\pi\cong\otimes\pi_v$ of $\GSp(4,\A_{\Q})$ (in the sense of \cite[Section~3]{Schmidt2017}). The automorphic representation $\pi$ associated to any such $f$ has trivial central character and hence may be viewed as an automorphic representation of  $\PGSp(4,\A_{\Q})$. Moreover, it follows from \cite[Section~2.1]{RoySchmidtYi2021} that such $\pi\cong\otimes\pi_v$ is an element of $S_k(p,\Omega)$, where $\pi_p$ is an Iwahori-spherical representation of type $\Omega$.  In fact, every eigenform $f\in S_k(\Gamma_p)$ arises from a vector in $\pi_p^{C_p}$. Here $C_p$ is one of the compact open subgroups in $\GSp(4, \Q_p)$ that corresponds to the congruence subgroup $\Gamma_p$ of $\Sp(4,\Q)$ given in \eqref{classical congrunece subgroups}. So, we have the following relationship between dimension of $S_k(\Gamma_p)$ and the quantities $s_k(p,\Omega)$
\begin{equation}\label{dimskreleq}
	\dim_{\C}S_k(\Gamma_p)= \sum\limits_{\Omega} \sum_{\pi  \in S_k(p,\Omega)} \dim\pi_p^{C_p}= \sum\limits_{\Omega} s_k(p,\Omega)\,d_{C_p,\Omega}.
\end{equation}
Here $d_{C_p,\Omega}$ is the common dimension  for  all Iwahori-spherical representations of type $\Omega$. These quantities are given explicitly in \cite[Table 3]{Schmidt200501}. For the purpose of this article, we will use \eqref{dimskreleq} to get some congruences for $\dim_{\C}S_k(\Gamma_p)$. In order to compute $s_k(p,\Omega)$ explicitly one needs to look at the representations in $S_k(p,\Omega)$ inside different Arthur packets. In particular, the Arthur packets types {\bf(G)}, {\bf(Y)} and {\bf(P)} are known as the general type, the Yoshida type and the Saito-Kurokawa type, respectively. See \cite{Arthur2004,Schmidt2018} for more details about the Arthur packets for $\GSp(4)$. Let $S_k^{(*)}(p,\Omega)$ be the set of those $\pi\in S_k(p,\Omega)$ that lie in an Arthur packet of type $(*)$ and $s_k^{(*)}(p,\Omega)=\#S_k^{(*)}(p,\Omega)$. Then we get
\begin{equation}\label{no. of cup form0}
	s_k(p,\Omega)=s_k^{\rm (\mathbf{G})}(p,\Omega)+s_k^{\rm (\mathbf{Y})}(p,\Omega)+s_k^{\rm (\mathbf{P})}(p,\Omega).
\end{equation}
A summary of the quantities $s_k^{(*)}(p,\Omega)$ that are considered in \eqref{dimskreleq} and their explicit formulas is given in \cite[Section~2.2]{RoySchmidtYi2021}. 

The Saito-Kurokawa type {\bf(P)} and the Yoshida type {\bf(Y)} are two kinds of liftings from elliptic cuspidal automorphic representations. As a consequence, both $s_k^{\textbf{(P)}}(p,\Omega)$ and $s_k^{\textbf{(Y)}}(p,\Omega)$ are related to dimensions of spaces of elliptic modular forms. The following lemma is useful for finding the quantities $s_k(p,\Omega)$ for the representations of Saito-Kurokawa type and Yoshida type. This result can be obtained from the work of \cite{Yamauchi1973},
and it is explicitly given in \cite[Theorem~2.2]{Martin2018}. Here, $S_{k}^{\rm new}(\Gamma_{0}^{(1)}(p))$ is the new subspace of weight $k$ elliptic cusp forms on the congruence subgroup $\Gamma_{0}^{(1)}(p)$ of $\SL(2,\Z)$. The plus and minus spaces $S_{k}^{\pm,\rm new}(\Gamma_{0}^{(1)}(p))$ are the space spanned by the eigenforms in $S_{k}^{\rm new}(\Gamma_{0}^{(1)}(p))$ which have the sign $\pm 1$ in the functional equation of their $L$-functions.

	\begin{lemma}
	\label{theorem for +,- new space}
	For $p>3$ and even $k\ge 2$,
	\begin{equation*}
		\dim_{\C} S_{k}^{\pm,{\rm new}}(\Gamma_{0}^{(1)}(p))=\frac 12\dim_{\C} S_{k}^{{\rm new}}(\Gamma_{0}^{(1)}(p))\pm \frac12 \left(\frac12 h(\Delta_p)b-\delta_{k,2}\right),
	\end{equation*}
	where $h(\Delta_p)$ is the class number of $\Q(\sqrt{-p})$ and  
	\begin{equation}
		\label{b}
		b=\begin{cases} 
			1 & \text{if } p \equiv 1 \pmod 4,\\
			2 & \text{if } p \equiv 7 \pmod 8,\\
			4 & \text{if } p \equiv 3 \pmod 8.\\
		\end{cases} 
		\hspace{0.5in} \text{and}	\hspace{0.5in}
		\delta_{k,2}=
		\begin{cases}
			1 & \text{if } k=2,\\
			0 & \text{if } k\neq 2.
		\end{cases}
	\end{equation}
	For $k>2$,
	\begin{align*}
		\dim_{\C} S_{k}^{\pm,{\rm new}}(\Gamma_{0}^{(1)}(2))=&\frac 12\dim_{\C} S_{k}^{{\rm new}}(\Gamma_{0}^{(1)}(2))\pm
		\begin{cases} 
			\frac12 & \text{if } k \equiv 0,2 \pmod 8,\\
			0 & \text{else}.
		\end{cases}\\
		\dim_{\C} S_{k}^{\pm,{\rm new}}(\Gamma_{0}^{(1)}(3))=&\frac 12\dim_{\C} S_{k}^{{\rm new}}(\Gamma_{0}^{(1)}(3))\pm
		\begin{cases} 
			\frac12 & \text{if } k \equiv 0,2,6,8 \pmod {12},\\
			0 & \text{else}.
		\end{cases}
	\end{align*}
\end{lemma}
We note that, when $\Omega$ is one of the types ${\rm IIb}, {\rm Vb}, {\rm VIc}$, there is a global representation in $S_k^{\rm (\mathbf{P})}(p,\Omega)$. When $\Omega$ is of type $ {\rm VIb}$, there are global representations in $S_k^{(*)}(p,\Omega)$ for all three types $(\mathbf{G})$, $(\mathbf{Y})$, and $ (\mathbf{P})$.

From now on we assume $k\ge 3$. Then, by Lemma~\ref{theorem for +,- new space} and \textcolor{black}{\cite[(3.10)]{RoySchmidtYi2021}}, for $p\ge 5$ we get the following identities
\begin{equation}\label{relations of Saito-Kurokawa type}
	\begin{split}
s_k(p,{\rm Vb})&=\begin{cases}
			0&\text{if } k \text{ is odd},\\
						\frac 12\dim_{\C} S_{2k-2}^{{\rm new}}(\Gamma_{0}^{(1)}(p))-\frac14 h(\Delta_p)b\qquad &\text{if } k \text{ is even}.
		\end{cases}\\
		s_k^{\textbf{(P)}}(p,{\rm VIb})&= \begin{cases}
			0&\text{if } k \text{ is odd},\\
		\frac 12\dim_{\C} S_{2k-2}^{{\rm new}}(\Gamma_{0}^{(1)}(p))+\frac14 h(\Delta_p)b \qquad &\text{if } k \text{ is even}.
		\end{cases}\\
		s_k(p,{\rm VIc})&=\begin{cases}
					\frac 12\dim_{\C} S_{2k-2}^{{\rm new}}(\Gamma_{0}^{(1)}(p))-\frac14 h(\Delta_p)b \qquad &\text{if } k \text{ is odd},\\
			0&\text{if } k \text{ is even}.
		\end{cases}
	\end{split}
	\end{equation}  
Similarly, \textcolor{black}{by Lemma~\ref{theorem for +,- new space} and \cite[(3.11)]{RoySchmidtYi2021}}, we have $s_k^{{\rm \textbf{(Y)}}}(p,{\rm VIb})=0$ for $p=2,3$, and for $p\ge5$ we have
\begin{equation}\label{Ykeven}	
\begin{split}
s_k^{{\rm \textbf{(Y)}}}(p,{\rm VIb})=\frac{1}{2}\dim_{\C} S_{2k-2}^{{\rm new}}(\Gamma_{0}^{(1)}(p))\dim_{\C} S_{2}^{{\rm new}}(\Gamma_{0}^{(1)}(p))
+\frac{(-1)^{k}}{8}h(\Delta_p)^2b^2-\frac{(-1)^k}{4}h(\Delta_p)b.
\end{split}
\end{equation} 

	\section{Congruences for dimensions of spaces of Siegel cusp forms}\label{main results}
In this section, we derive some congruences modulo $16$ and modulo $4$ involving dimensions of spaces of Siegel cusp forms of degree $2$,  the class number of $\Q(\sqrt{-p})$, and dimensions of spaces of elliptic modular newforms.
	\begin{theorem}\label{Proposition of IIIa/VI p geq 5}
		Let $h(\Delta_p)$ be the class number of $\Q(\sqrt{-p})$. For $k\ge 3$ and $p \ge 5$, we have the following congruence relations.
		
		\rm{(i)}	For $p \equiv 1 \pmod 4$
		\begin{equation*}
			\begin{split}
				&(-1)^{k-1}h(\Delta_p)^2-\begin{cases}4h(\Delta_p)&\text{if } k \text{ is odd}\\0&\text{if } k \text{ is even}\end{cases}\\
				&\equiv 4\left(\dim_{\C} S_{2k-2}^{{\rm new}}(\Gamma_{0}^{(1)}(p))\dim_{\C} S_{2}^{{\rm new}}(\Gamma_{0}^{(1)}(p))-(-1)^{k-1}\dim_{\C} S_{2k-2}^{{\rm new}}(\Gamma_{0}^{(1)}(p))\right)\\
				&\hspace{30ex}+8\Big(\dim_{\C} S_k({\rm K}(p))-\dim_{\C} S_k(\Gamma_0(p))\Big)\quad \pmod{16}.
			\end{split}
		\end{equation*}
		
		\rm{(ii)}	For $p \equiv 7 \pmod 8$
		\begin{equation*}
			\begin{split}
				&(-1)^{k-1}h(\Delta_p)^2-\begin{cases}2h(\Delta_p)&\text{if } k \text{ is odd}\\0&\text{if } k \text{ is even}\end{cases}\\
				&\equiv \dim_{\C} S_{2k-2}^{{\rm new}}(\Gamma_{0}^{(1)}(p))\dim_{\C} S_{2}^{{\rm new}}(\Gamma_{0}^{(1)}(p))-(-1)^{k-1}\dim_{\C} S_{2k-2}^{{\rm new}}(\Gamma_{0}^{(1)}(p))\\
				&\hspace{31ex} +2\Big(\dim_{\C} S_k({\rm K}(p))-\dim_{\C} S_k(\Gamma_0(p))\Big) \quad\pmod{4}.
			\end{split}
		\end{equation*}
		
		\rm{(iii)}	For $p \equiv 3 \pmod 8$
		\begin{equation*}
			\begin{split}
				&2\left(\dim_{\C} S_k({\rm K}(p))-\dim_{\C} S_k(\Gamma_0(p))\right) \\
				&\equiv (-1)^{k-1}\dim_{\C} S_{2k-2}^{{\rm new}}(\Gamma_{0}^{(1)}(p))-\dim_{\C} S_{2k-2}^{{\rm new}}(\Gamma_{0}^{(1)}(p))\dim_{\C} S_{2}^{{\rm new}}(\Gamma_{0}^{(1)}(p)) \qquad\pmod{4}.
			\end{split}
		\end{equation*}
	\end{theorem}
	\begin{proof} 
	\textcolor{black}{From \cite[(2.5) and (2.6)]{RoySchmidtYi2021}, we have 
	\begin{equation}
	\label{eqIIIa.1}
	    \begin{split}
	        s_k^{\rm \mathbf{(G)}}(p, {\rm VIa/b})&=s_k^{\rm \mathbf{(G)}}(p, {\rm VIa})=s_k^{\rm \mathbf{(G)}}(p, {\rm VIb}),\\
	        s_k(p, {\rm IIIa+VIa/b})&:= s_k^{\rm \mathbf{(G)}}(p, {\rm IIIa})+s_k^{\rm \mathbf{(G)}}(p, {\rm VIa/b}).
	    \end{split}
	\end{equation}
Using \eqref{eqIIIa.1} and replacing $s_k(p,{\rm VIb})$ by $s_k^{\mathbf{(G)}}(p,{\rm VIb})+s_k^{\mathbf{(Y)}}(p,{\rm VIb})+s_k^{\mathbf{(P)}}(p,{\rm VIb})$ in \cite[(3.12)]{RoySchmidtYi2021}, we obtain the following identity
}
		\begin{equation}\label{IIIaVIab}
			\begin{split}
			s_k(p,{\rm IIIa+VIa/b})&=\frac 12 \dim_{\C} S_k(\Gamma_0(p))-\frac 12 \dim_{\C} S_k({\rm K}(p))-s_k(p,{\rm I})-s_k(p,{\rm IIb})\\
			&\qquad -\frac 12 s_k^{\rm \mathbf{(P)}}(p,{\rm VIb})-\frac 12 s_k^{\rm \mathbf{(Y)}}(p,{\rm VIb})+ \frac 12 s_k(p,{\rm VIc}).
			\end{split}
		\end{equation}	
		Using \eqref{relations of Saito-Kurokawa type} and \eqref{Ykeven} we obtain
		\begin{align*}
			&2(s_k(p,{\rm IIIa+VIa/b})+s_k(p,{\rm I})+s_k(p,{\rm IIb}))+\dim_{\C} S_k({\rm K}(p))-\dim_{\C} S_k(\Gamma_0(p))\\
			&=(-1)^{k-1}\frac 12\dim_{\C} S_{2k-2}^{\rm new}(\Gamma_{0}^{(1)}(p))-\frac{1}{2}\dim_{\C} S_{2k-2}^{\rm new}(\Gamma_{0}^{(1)}(p))\dim_{\C} S_{2}^{\rm new}(\Gamma_{0}^{(1)}(p))\\
			&\qquad  +(-1)^{k-1}\frac{1}{8}h(\Delta_p)^2b^2-\begin{cases}\frac{1}{2}h(\Delta_p)b&\mbox{if $k$ is odd},\\0&\mbox{if $k$ is even}.\end{cases}
		\end{align*} 
		\begin{enumerate}
			\item[(i)] If $p\equiv 1\ (\mathrm{mod}\ 4)$, we have $b=1$. Then we get
			\begin{equation*}
				\begin{split}
					&16(s_k(p,{\rm IIIa+VIa/b})+s_k(p,{\rm I})+s_k(p,{\rm IIb}))+8\dim_{\C} S_k({\rm K}(p))-8\dim_{\C} S_k(\Gamma_0(p))\\
					&=4(-1)^{k-1}\dim_{\C} S_{2k-2}^{\rm new}(\Gamma_{0}^{(1)}(p))-4\dim_{\C} S_{2k-2}^{\rm new}(\Gamma_{0}^{(1)}(p))\dim_{\C} S_{2}^{\rm new}(\Gamma_{0}^{(1)}(p))\\
			&\qquad  +(-1)^{k-1}h(\Delta_p)^2-\begin{cases}4h(\Delta_p)&\mbox{if $k$ is odd},\\0&\mbox{if $k$ is even}.\end{cases}
				\end{split}
			\end{equation*}
			\item[(ii)]  If $p\equiv 7\ (\mathrm{mod}\ 8)$, we have $b=2$. Then we get
			\begin{equation*}
				\begin{split}
					&4(s_k(p,{\rm IIIa+VIa/b})+s_k(p,{\rm I})+s_k(p,{\rm IIb}))+2\dim_{\C} S_k({\rm K}(p))-2\dim_{\C} S_k(\Gamma_0(p))\\
					&=(-1)^{k-1}\dim_{\C} S_{2k-2}^{\rm new}(\Gamma_{0}^{(1)}(p))-\dim_{\C} S_{2k-2}^{\rm new}(\Gamma_{0}^{(1)}(p))\dim_{\C} S_{2}^{\rm new}(\Gamma_{0}^{(1)}(p))\\
		&\qquad  +(-1)^{k-1}h(\Delta_p)^2-\begin{cases}2h(\Delta_p)&\mbox{if $k$ is odd},\\0&\mbox{if $k$ is even}.\end{cases}
				\end{split}
			\end{equation*}
			\item[(iii)]  If $p\equiv 3\ (\mathrm{mod}\ 8)$, we have $b=4$. Then we get
			\begin{equation*}
				\begin{split}
					&4(s_k(p,{\rm IIIa+VIa/b})+s_k(p,{\rm I})+s_k(p,{\rm IIb}))+2\dim_{\C} S_k({\rm K}(p))-2\dim_{\C} S_k(\Gamma_0(p))\\
					&=(-1)^{k-1}\dim_{\C} S_{2k-2}^{\rm new}(\Gamma_{0}^{(1)}(p))-\dim_{\C} S_{2k-2}^{\rm new}(\Gamma_{0}^{(1)}(p))\dim_{\C} S_{2}^{\rm new}(\Gamma_{0}^{(1)}(p))\\
			&\qquad  +(-1)^{k-1}\cdot4h(\Delta_p)^2-\begin{cases}4h(\Delta_p)&\mbox{if $k$ is odd},\\0&\mbox{if $k$ is even}.\end{cases}
				\end{split}
			\end{equation*}
		\end{enumerate} 
		The desired congruences now follow from above discussions. 
	\end{proof}
	
	\begin{theorem}\label{Proposition of Va p geq 5}
		Let $h(\Delta_p)$ be the class number of $\Q(\sqrt{-p})$. For $k\ge 3$ and $p \ge 5$, we have the following congruence relations.
		
		\rm{(i)}	For $p \equiv 1 \pmod 4$
		\begin{equation*}
			\begin{split}
				(-1)^{k}h(\Delta_p)^2
				&\equiv -4\Big(\dim_{\C} S_{2k-2}^{{\rm new}}(\Gamma_{0}^{(1)}(p))\dim_{\C} S_{2}^{{\rm new}}(\Gamma_{0}^{(1)}(p))+\dim_{\C} S_{2k-2}^{{\rm new}}(\Gamma_{0}^{(1)}(p))\Big)\\
				&\hspace{22ex} +8\Big(\dim_{\C} S_k({\rm K}(p))+\dim_{\C} S_k(\Gamma_0(p))\Big) \pmod{16}.
			\end{split}
		\end{equation*}
		
		\rm{(ii)}	For $p \equiv 7 \pmod 8$
		\begin{equation*}
			\begin{split}
				(-1)^{k}h(\Delta_p)^2
				&\equiv -\dim_{\C} S_{2k-2}^{{\rm new}}(\Gamma_{0}^{(1)}(p))\dim_{\C} S_{2}^{{\rm new}}(\Gamma_{0}^{(1)}(p))-\dim_{\C} S_{2k-2}^{{\rm new}}(\Gamma_{0}^{(1)}(p))\\
				&\hspace{22ex}+2\Big(\dim_{\C} S_k({\rm K}(p))+\dim_{\C} S_k(\Gamma_0(p))\Big) \pmod{4}.
			\end{split}
		\end{equation*}	
		
		\rm{(iii)}	For $p \equiv 3 \pmod 8$
		\begin{equation*}
			\begin{split}
				&2\left(\dim_{\C} S_k({\rm K}(p))+\dim_{\C} S_k(\Gamma_0(p))\right) \\
				& \equiv \dim_{\C} S_{2k-2}^{{\rm new}}(\Gamma_{0}^{(1)}(p))+\dim_{\C} S_{2k-2}^{{\rm new}}(\Gamma_{0}^{(1)}(p))\dim_{\C} S_{2}^{{\rm new}}(\Gamma_{0}^{(1)}(p))  \pmod{4}.
			\end{split}
		\end{equation*}
	\end{theorem}
	\begin{proof} \textcolor{black}{We have the following formula for $s_k(p,{\rm Va})$ from the proof of \cite[Theorems 3.5–3.8]{RoySchmidtYi2021}
	}
		\begin{equation*}
			\begin{split}
				&s_k(p,{\rm Va})=\dim_{\C} S_k(\Kl(p))-\frac 12 \dim_{\C} S_k(\Gamma_0(p))-\frac 32\dim_{\C} S_k({\rm K}(p))+s_k(p,{\rm I})\nonumber\\
				&\hspace{5ex}+s_k(p,{\rm IIb})+s_k(p,{\rm Vb})+\frac 12 s_k^{\rm \mathbf{(P)}}(p,{\rm VIb})+\frac 12 s_k^{\rm \mathbf{(Y)}}(p,{\rm VIb})+\frac 12 s_k(p,{\rm VIc}).
			\end{split}
		\end{equation*}
		Using \eqref{relations of Saito-Kurokawa type} and \eqref{Ykeven} we obtain
		\begin{align*}
			&2(s_k(p,{\rm Va})-s_k(p,{\rm I})-s_k(p,{\rm IIb})-s_k(p,{\rm Vb}))\\
			 &-2\dim_{\C} S_k(\Kl(p))+\dim_{\C} S_k(\Gamma_0(p))+3\dim_{\C} S_k({\rm K}(p))\\
			&=\frac 12\dim_{\C} S_{2k-2}^{{\rm new}}(\Gamma_{0}^{(1)}(p))+\frac{1}{2}\dim_{\C} S_{2k-2}^{{\rm new}}(\Gamma_{0}^{(1)}(p))\dim_{\C} S_{2}^{{\rm new}}(\Gamma_{0}^{(1)}(p))+(-1)^k\frac{1}{8}h(\Delta_p)^2b^2.
		\end{align*} 
		Then, by the same arguments as in Theorem~\ref{Proposition of IIIa/VI p geq 5}, we get the desired congruences.
	\end{proof}
Similarly, if we use the following identity \textcolor{black}{from the proof of \cite[Theorems 3.5–3.8]{RoySchmidtYi2021}}
\begin{equation*}
s_k(p,{\rm IIa})=\dim_{\C} S_k({\rm K}(p))-2s_k(p,{\rm I})-s_k(p,{\rm IIb})-s_k(p,{\rm Vb})-s_k(p,{\rm VIc}),
\end{equation*}
we obtain a congruence that relates dimension of the space $S_{2k-2}^{{\rm new}}(\Gamma_{0}^{(1)}(p))$ of elliptic newforms and the class number of $\Q(\sqrt{-p})$. This result also follows from \cite[Theorem~3.3]{Wakatsuki2014}. 
	\begin{theorem}\label{Theorem 3.3 in Wakatsuki2014}
		Let $h(\Delta_p)$ be the class number of $\Q(\sqrt{-p})$. For $k\ge 3$ and $p \ge 5$, we have the following congruences	
		\begin{equation*}
			\begin{split}
				\dim_{\C} S_{2k-2}^{{\rm new}}(\Gamma_{0}^{(1)}(p)) \equiv \frac 12 h(\Delta_p)b=\begin{cases}  \frac 12 h(\Delta_p) \pmod{2}&\text{if } p \equiv 1 \pmod 4,\\
					h(\Delta_p)\ \, \pmod{2}&\text{if } p \equiv 7 \pmod 8,\\ 0\qquad\ \, \pmod{2}&\text{if } p \equiv 3 \pmod 8.\end{cases}
			\end{split}
		\end{equation*}
	\end{theorem}
Next we consider the results for $p=2, 3$. We have to treat $p=2,3$ separately because the formulas in Lemma~\ref{theorem for +,- new space} are different for these two primes.
	\begin{proposition}\label{Proposition for IIIa/VI(2)}
	Suppose $k\geq 3$. Then, modulo 4 we get
	
	\begin{align*}
					2\dim_{\C} S_k({\rm K}(2))-2\dim_{\C} S_k(\Gamma_0(2))
				&\equiv\begin{cases}
						-\dim_{\C} S_{2k-2}^{{\rm new}}(\Gamma_{0}^{(1)}(2)) &\text{if } k \equiv 0\pmod 4,\\[1ex]
						\dim_{\C} S_{2k-2}^{{\rm new}}(\Gamma_{0}^{(1)}(2))+1 &\text{if } k \equiv 1\pmod 4,\\[1ex]
						-\dim_{\C} S_{2k-2}^{{\rm new}}(\Gamma_{0}^{(1)}(2))-1 &\text{if } k \equiv 2\pmod 4,\\[1ex]
						\dim_{\C} S_{2k-2}^{{\rm new}}(\Gamma_{0}^{(1)}(2)) &\text{if } k \equiv 3\pmod 4,
				\end{cases}
		\end{align*}
and		
		\begin{align*}
	2\dim_{\C} S_k({\rm K}(3))-2\dim_{\C} S_k(\Gamma_0(3))
	&\equiv\begin{cases}
		-\dim_{\C} S_{2k-2}^{{\rm new}}(\Gamma_{0}^{(1)}(3)) &\text{if } k \equiv 0 \hspace{2ex}\pmod 6,\\[1ex]
		\dim_{\C} S_{2k-2}^{{\rm new}}(\Gamma_{0}^{(1)}(3))+1 &\text{if } k \equiv 1,5\pmod 6,\\[1ex]
		-\dim_{\C} S_{2k-2}^{{\rm new}}(\Gamma_{0}^{(1)}(3))-1 &\text{if } k \equiv 2,4\pmod 6,\\[1ex]
		\dim_{\C} S_{2k-2}^{{\rm new}}(\Gamma_{0}^{(1)}(3)) &\text{if } k \equiv 3\hspace{2ex}\pmod 6.
	\end{cases}	
\end{align*}				
\end{proposition}

	\begin{proof}
		Since $s_k^{\rm \mathbf{(Y)}}(p,{\rm VIb})=0$ for $p=2, 3$, from \eqref{IIIaVIab} we get
		\begin{equation}\label{IIIa/VI(23)}
			\begin{split}
			&2(s_k(p, {\rm IIIa+VIa/b})+s_k(p,{\rm I})+s_k(p, {\rm IIb}))+\dim_{\C} S_k({\rm K}(p))-\dim_{\C} S_k(\Gamma_0(p))\\
				&=s_k(p, {\rm VIc})- s_k^{\rm \mathbf{(P)}}(p,{\rm VIb}).
			\end{split}
		\end{equation}
		If $k$ is odd, it follows from  Lemma~\ref{theorem for +,- new space} \textcolor{black}{and \cite[(3.10)]{RoySchmidtYi2021}} that
		\begin{align*}
s_k(2, {\rm VIc})=&
\frac{1}{2}\dim_{\C} S_{2k-2}^{{\rm new}}(\Gamma_{0}^{(1)}(2))+\begin{cases} 
				\frac12 & \text{if } 2k-2 \equiv 0 \pmod 8,\\
				0 & \text{else}.
			\end{cases}\\
			=&\frac{1}{2}\dim_{\C} S_{2k-2}^{{\rm new}}(\Gamma_{0}^{(1)}(2))+\begin{cases} 
				\frac12 & \text{if } k \equiv 1 \pmod 4,\\
				0 & \text{if } k \equiv 3 \pmod 4.
			\end{cases}\\
			s_k(3, {\rm VIc})=&
			\frac{1}{2}\dim_{\C} S_{2k-2}^{{\rm new}}(\Gamma_{0}^{(1)}(3))+
			\begin{cases} 
				\frac12 & \text{if } 2k-2 \equiv 0,8 \pmod {12},\\
				0 & \text{else}.
			\end{cases}\\
			=&\frac{1}{2}\dim_{\C} S_{2k-2}^{{\rm new}}(\Gamma_{0}^{(1)}(3))+
			\begin{cases} 
				\frac12 & \text{if } k \equiv 1,5 \pmod {6},\\
				0 & \text{if } k \equiv 3 \pmod 6.
			\end{cases}
		\end{align*}
		If $k$ is even, by Lemma~\ref{theorem for +,- new space} \textcolor{black}{ and \cite[(3.10)]{RoySchmidtYi2021}} we obtain
		\begin{align*}
s_k^{\rm \mathbf{(P)}}(2,{\rm VIb})=&
\frac 12\dim_{\C} S_{2k-2}^{{\rm new}}(\Gamma_{0}^{(1)}(2))+
			\begin{cases} 
				\frac12 & \text{if } 2k-2 \equiv 2 \pmod 8,\\
				0 & \text{else}.
			\end{cases}\\
			=&\frac 12\dim_{\C} S_{2k-2}^{{\rm new}}(\Gamma_{0}^{(1)}(2))+
			\begin{cases} 
				\frac12 & \text{if } k \equiv 2 \pmod 4,\\
				0 & \text{if } k \equiv 0 \pmod 4.
			\end{cases}\\
			s_k^{\rm \mathbf{(P)}}(3,{\rm VIb})=&\frac 12\dim_{\C} S_{2k-2}^{{\rm new}}(\Gamma_{0}^{(1)}(3))+
			\begin{cases} 
				\frac12 & \text{if } 2k-2 \equiv 2,6 \pmod {12},\\
				0 & \text{else}.
			\end{cases}\\
			=&\frac 12\dim_{\C} S_{2k-2}^{{\rm new}}(\Gamma_{0}^{(1)}(3))+
			\begin{cases} 
				\frac12 & \text{if } k \equiv 2,4 \pmod {6},\\
				0 & \text{if } k \equiv 0 \pmod {6}.
			\end{cases}
		\end{align*}
Then we get the following 
			\begin{align*}
				&4(s_k(2,{\rm IIIa+VIa/b})+s_k(2,{\rm I})+s_k(2,{\rm IIb}))+2\dim_{\C} S_k({\rm K}(2))-2\dim_{\C} S_k(\Gamma_0(2))\\
				&=\begin{cases}
						-\dim_{\C} S_{2k-2}^{{\rm new}}(\Gamma_{0}^{(1)}(2)) &\text{if } k \equiv 0\pmod 4,\\[1ex]
						\dim_{\C} S_{2k-2}^{{\rm new}}(\Gamma_{0}^{(1)}(2))+1 &\text{if } k \equiv 1\pmod 4,\\[1ex]
						-\dim_{\C} S_{2k-2}^{{\rm new}}(\Gamma_{0}^{(1)}(2))-1 &\text{if } k \equiv 2\pmod 4,\\[1ex]
						\dim_{\C} S_{2k-2}^{{\rm new}}(\Gamma_{0}^{(1)}(2)) &\text{if } k \equiv 3\pmod 4,
				\end{cases}	
			\end{align*}

and
\begin{align*}
	&4(s_k(3,{\rm IIIa+VIa/b})+s_k(3,{\rm I})+s_k(3,{\rm IIb}))+2\dim_{\C} S_k({\rm K}(3))-2\dim_{\C} S_k(\Gamma_0(3))\\
	&=\begin{cases}
		-\dim_{\C} S_{2k-2}^{{\rm new}}(\Gamma_{0}^{(1)}(3)) &\text{if } k \equiv 0\hspace{2ex}\pmod 6,\\[1ex]
		\dim_{\C} S_{2k-2}^{{\rm new}}(\Gamma_{0}^{(1)}(3))+1 &\text{if } k \equiv 1,5\pmod 6,\\[1ex]
		-\dim_{\C} S_{2k-2}^{{\rm new}}(\Gamma_{0}^{(1)}(3))-1 &\text{if } k \equiv 2,4\pmod 6,\\[1ex]
		\dim_{\C} S_{2k-2}^{{\rm new}}(\Gamma_{0}^{(1)}(3)) &\text{if } k \equiv 3\hspace{2ex}\pmod 6.
	\end{cases}	
\end{align*}
		Hence the stated congruences follow from above equations.
	\end{proof}
	\section{Congruences for \texorpdfstring{$4$}{}-core partition functions}\label{Section 4core}
	A partition of a positive integer $n$ is any non-increasing sequence of positive integers  whose sum is $n$. If $\pi=\pi_1+\pi_2+\cdots +\pi_m$ be a partition of $n$, where $\pi_1\geq \pi_2 \geq \cdots \geq \pi_m$, then the
	Ferrers–Young diagram of $\pi$ is an array of nodes with $\pi_p$ nodes in the $p$-th row.
	
	\begin{align*}
		&\bullet~~\bullet~~\cdots~~\bullet~~~\bullet & \pi_1~\text{nodes}\\
		&\bullet~~\bullet~~\cdots~~\bullet & \pi_2~\text{nodes}\\
		&~\vdots &~\vdots\\
		&\bullet~~\cdots~~\bullet & \pi_m~\text{nodes}
	\end{align*}

	The $(p,q)$-hook is the set of nodes directly below and directly to the right of the $(p, q)$-node, as well as the $(p, q)$-node.
	The hook number, $H(p,q)$, is the total number of nodes on the $(p,q)$-hook. For a positive integer $t$, a $t$-core partition of $n$ is a partition of $n$ in which none of the hook numbers are divisible by $t$. Suppose $c_t(n)$ denote the number of $t$-core partitions of $n$.
	Now we illustrate the Ferrers–Young diagram of the partition $5 + 4 + 2$ of $11$ with the corresponding  hook numbers shown in the graph.
	\begin{align*}
		&\bullet^7~~\bullet^6~~\bullet^4~~\bullet^3~~\bullet^1\\
		&\bullet^5~~\bullet^4~~\bullet^2~~\bullet^1\\ &\bullet^2~~\bullet^1
	\end{align*}
	Therefore it is clear that if $t>7$ then $5+4+2$ is a $t$-core partition of $11$. The generating function for $c_t(n)$ is
	\begin{align*}
		\sum_{n=0}^{\infty}c_t(n)q^n:= \prod_{n=1}^{\infty}\frac{(1-q^{tn})^t}{(1-q^n)}.
	\end{align*}
	In this article we focus on $4$-core partitions. Hirschhorn and Sellers \cite{Hirschhorn1996EJC}, proved that 
	\begin{align*}
		c_{4}\left(9n+2\right)&\equiv0 \pmod{2},\\
		c_{4}\left(9n+8\right)&\equiv0 \pmod{2}.
	\end{align*}
	Analyzing the
	action of the Hecke operators on the space of integer weight cusp forms Boylan \cite{Boylan2002} proved that 
	\begin{align*}
		c_{4}\left(Bn-5\right)&\equiv0 \pmod{2},
	\end{align*}
	where $B$ is product of any five distinct odd primes. Regarding the positivity of $4$-core partition, Ono \cite{Ono1995} proved that $c_4(n)>0$ for all $n\geq 0.$
	
	It is well known that the number of $4$-core  partitions $c_4(n)$ of $n$ is equal to the number of representations of $8n+5$ in the form $x^2+2y^2+2z^2$ with $x,y,z$ odd  positive integers.
	Let $H(D)$ denote the Hurwitz class numbers of binary quadratic forms of discriminant $D<0$. There is relationship between the Hurwitz class numbers and the class number of imaginary quadratic fields. Let $h(\Delta_n)$ be the class number for imaginary quadratic fields $\Q(\sqrt{-n})$, where $\Delta_n<0$ is the discriminant of $\Q(\sqrt{-n})$. Here $\Delta_n$ is a fundamental discriminant given by
	$$\Delta_n=\begin{cases}
		-4n& \text{if $n \equiv 1,2 \pmod {4},$} \\
		-n& \text{if $n \equiv 3 \pmod {4}$}. \\
	\end{cases}$$
	Suppose $D=\Delta_nf^2$ where $\Delta_n<0$ is a fundamental discriminant for a squarefree integer $n$. Then we have 
	\begin{equation*}
		H(D)=\frac{2h(\Delta_n)}{w(\Delta_n)}\sum_{d\mid f}\mu(d)\left(\frac{\Delta_n}{d}\right)\sigma_1(f/d),
	\end{equation*}
	where 
	$w(\Delta_n)$ is the number
	of units in the ring of integers of $\Q(\sqrt{-n})$, $\mu(n)$ is the M\"obius function and $\sigma_1(n)$ is the sum of the
	divisors of $n$. In particular, when $D<0$ is a fundamental discriminant and $D\neq 3,4$ then we have $H(D)=h(D)$. There is very interesting connection between $4$-core partition functions and the Hurwitz class numbers.
	
	In 1997, Ono and Sze \cite{Ono1997} showed that $4$-core  partitions naturally arise in algebraic number theory. In particular they proved that, if $8n + 5$ is square-free integer, then
	\begin{align}	
		\label{c4n-classnumber}
		c_4(n)&=\frac{1}{2}h(-32n-20).	
	\end{align}
	Note that 
	by \emph{Dirichlet's Theorem on primes in arithmetic progression}, there are infinitely many primes of the form $8n+5$. 
	Next, we prove a congruence for $c_4(n)$ using the results from the previous section.
	\begin{corollary}
	    \label{Theorem for C4n and Smf}
		Suppose $8n+5$ is a prime number and $k\ge 5$. Then, modulo $4$, we have
		\begin{equation*}
			\begin{split}
				&(-1)^{k-1}c_4(n)^2-\begin{cases}2c_4(n)&\text{if } k \text{ is odd}\\0&\text{if } k \text{ is even}\end{cases}-2\Big(\dim_{\C} S_k({\rm K}(8n+5))-\dim_{\C} S_k(\Gamma_0(8n+5))\Big)\\
				&\equiv \left(\dim_{\C} S_{2k-2}^{\rm new}(\Gamma_{0}^{(1)}(8n+5))\dim_{\C} S_{2}^{\rm new}(\Gamma_{0}^{(1)}(8n+5))-(-1)^{k-1}\dim_{\C} S_{2k-2}^{\rm new}(\Gamma_{0}^{(1)}(8n+5))\right)\\		
				&\text{ and }	\\	
				&	(-1)^{k}c_4(n)^2-2\Big(\dim_{\C} S_k({\rm K}(8n+5))-\dim_{\C} S_k(\Gamma_0(8n+5))\Big)\\
				& \equiv -\left(\dim_{\C} S_{2k-2}^{\rm new}(\Gamma_{0}^{(1)}(8n+5))\dim_{\C} S_{2}^{\rm new}(\Gamma_{0}^{(1)}(8n+5))+\dim_{\C} S_{2k-2}^{\rm new}(\Gamma_{0}^{(1)}(8n+5))\right).
			\end{split}
		\end{equation*}
	\end{corollary}
	\begin{proof}
		Let $p$ be a prime of the form $8n+5$. \textcolor{black}{In particular,} $p \equiv 1 \pmod{4}$.  It follows from \eqref{c4n-classnumber} that the class number of $\Q(\sqrt{-p})$ is given by $h(\Delta_p)=h(-32n-20)=2c_4(n)$. Then the first congruence
		follows from part (i) of Theorem~\ref{Proposition of IIIa/VI p geq 5} and the second congruence follows from  part (i) of Theorem~\ref{Proposition of Va p geq 5}.
	\end{proof}
\noindent Finally, by adding two different congruences in Corollary~\ref{Theorem for C4n and Smf} when $k$ is odd, we get the following result.	
	\begin{corollary}\label{coro4.2}
		Suppose $8n+5$ is a prime number and $k$ is a positive integer. Then 
		\begin{equation*}
			\begin{split}
				c_4(n) \equiv \dim_{\C} S_{4k}^{\rm new}(\Gamma_{0}^{(1)}(8n+5)) \pmod{2}.
			\end{split}
		\end{equation*}
	\end{corollary}
	
\thanks{\textbf{Acknowledgements.}} We would like to thank Ken Ono for helpful advice.	The first author has carried out this work at Harish-Chandra Research Institute, affiliated with Homi Bhabha National Institute (Department of Atomic Energy, India), as a Postdoctoral Fellow. We would also like to thank the referee for the detailed comments and suggestions.
	
	\bibliographystyle{plain}
	\bibliography{CongruenceDimensionSiegel.bib}
	
\end{document}